\documentclass[11pt,amsfonts]{article}
\usepackage{graphicx}
\usepackage{latexsym}
\usepackage{amssymb}
\usepackage{amsmath}
\usepackage{enumerate}
\usepackage{color}
\usepackage{layout}
\usepackage{bbm}
\newtheorem{proposition}{Proposition}
\newtheorem{lemma}{Lemma}
\newtheorem{definition}{Definition}
\newtheorem{corollary}{Corollary}
\newtheorem{theorem}{Theorem}
\newtheorem{remark}{Remark}

\def\real{{\mathord{{\rm I\kern-2.8pt R}}}}        
\def\inte{{\mathord{{\rm I\kern-2.8pt N}}}}

\def\sZZ{{\rm Z\kern-2.8ptem{}Z}}

\def\z{{\mathchoice
		{\sZZ}
		{\sZZ}
		{\rm Z\kern-0.30em{}Z}
		{\rm Z\kern-0.25em{}Z} }}
\def\sQQ{{\kern 0.27em \vrule height1.45ex width0.03em depth0em
		\kern-0.30em \rm Q}}
\def\qu{{\mathchoice
		{\sQQ}
		{\sQQ}
		{\kern 0.225em \vrule height1.05ex width0.025em depth0em \kern-0.25em \rm Q}
		{\kern 0.180em \vrule height0.78ex width0.020em depth0em \kern-0.20em \rm Q}
}}
\def\sCC{{\kern 0.27em \vrule height1.45ex width0.03em depth0em
		\kern-0.30em \rm C}}
\def\complex{{\mathchoice
		{\sCC}
		{\sCC}
		{\kern 0.225em \vrule height1.05ex width0.025em depth0em \kern-0.25em \rm C}
		{\kern 0.180em \vrule height0.78ex width0.020em depth0em \kern-0.20em \rm C}
}}

\newcommand{\ba}{\begin{array}}
	\newcommand{\ea}{\end{array}}
\newcommand{\be}{\begin{equation}}
	\newcommand{\ee}{\end{equation}}
\newcommand{\bea}{\begin{eqnarray}}
	\newcommand{\eea}{\end{eqnarray}}
\newcommand{\beaa}{\begin{eqnarray*}}
	\newcommand{\eeaa}{\end{eqnarray*}}

\def\z{\zeta}

\font\tenmath=msbm10 \font\sevenmath=msbm7 \font\fivemath=msbm5
\newfam\mathfam \textfont\mathfam=\tenmath
\scriptfont\mathfam=\sevenmath \scriptscriptfont\mathfam=\fivemath

\def \={{\buildrel {\rm (law)} \over =}}

\newcommand\HHH{\mathfrak{H}}

\def\qed{ \hfill \vrule width.25cm height.25cm depth0cm\smallskip}

\newcommand{\basa}{\begin{assumption}}
	\newcommand{\easa}{\end{assumption}}

\newcommand{\bas}{\begin{assum}}
	\newcommand{\eas}{\end{assum}}

\def\span{\hbox{\rm span$\,$}}

\newcommand{\ignore}[1]{}
\begin{document}
	
	\renewcommand{\thefootnote}{\fnsymbol{footnote}}
	
	\renewcommand{\thefootnote}{\fnsymbol{footnote}}

	\title{Malliavin smoothness of the Rosenblatt process}
	
	\author{
  Laurent Loosveldt\thanks{Université de Liège, Département de Mathématiques, Liège, Belgium}%
  \and
  Yassine Nachit\thanks{Université de Lille, Laboratoire Paul Painlevé, Lille, France}%
  \and
    Ivan Nourdin\thanks{University of Luxembourg, Department of Mathematics, Esch-sur-Alzette, Luxembourg}%
  \and
  Ciprian Tudor\thanks{Université de Lille, Laboratoire Paul Painlevé, Lille, France}%
}
	
	\maketitle

	\begin{abstract}
		We investigate the smoothness of the densities of the finite-dimensional distributions of the Rosenblatt process. Within the Malliavin calculus framework, we prove that Rosenblatt random vectors are nondegenerate in the Malliavin sense. As a consequence, their densities belong to the Schwartz space of rapidly decreasing smooth functions. The proof relies on establishing the existence of all negative moments of the determinant of the Malliavin matrix, exploiting the specific structure of random variables in the second Wiener chaos. In addition, we derive exponential-type upper bounds for the partial derivatives of the densities of the finite-dimensional distributions of the Rosenblatt process.
	\end{abstract}

	\vskip0.3cm


{\bf 2020 AMS Classification Numbers: } 60G22; 60G18; 60H07; 60F05.

\section{Introduction}
Let $F=\left( F_{1}, \ldots , F_{m}\right)$ be a random vector whose components are differentiable in the Malliavin sense. Denote by $ \Gamma_{F}$ its Malliavin matrix, whose entries are given by
	$\Gamma_{F}(i,j)= \langle DF_{i}, DF_{j}\rangle _{\HHH}$, $1\leq i, j\leq m$,
where $D$ denotes the Malliavin derivative with respect to an isonormal process $(W(h), h\in \HHH)$ on a real and separable Hilbert space $\HHH$. A classical result in Malliavin calculus (the {\it Bouleau-Hirsch criterion}) states that if $\Gamma_{F}$ is almost surely  invertible, then  $F$ admits a density with respect to the Lebesgue measure on $\mathbb{R} ^{m}$. 

The random vector $F$ is said to be nondegenerate if its components are infinitely Malliavin differentiable and if $\left( \det \Gamma_{F}\right) ^{-1}\in L ^{p}(\Omega)$ for every $p\geq 1$. In this case, its density is smooth and belongs to the Schwartz space of rapidly decreasing $C ^{\infty}$ functions on $ \mathbb{R} ^{m}$. Hence, obtaining a smooth density is tightly linked to the existence of negative moments of the Malliavin derivatives of the components. 

When the components of $F$ lie in a fixed Wiener chaos, the above result can be strengthened. It was shown in \cite{NNP} that in this setting the almost sure invertibility of $\Gamma_F$ is actually equivalent to the absolute continuity of the law of $F$.  Moreover, \cite{NT} proves that a pair of multiple stochastic integrals fails to admit a density if and only if its components are proportional, mirroring the Gaussian case.

In the recent work \cite{LNNT} we applied the Bouleau-Hirsch criterion to establish the absolute continuity of finite-dimensional distributions of Hermite processes. Let $q\geq 1$ be an integer and $ (Z ^{H, q}_{t}, t\in \mathbb{R})$ a Hermite process of order $q$ with self-similarity index $H\in \left( \frac{1}{2}, 1\right)$. Recall that $Z^{H,q}$ is an $H$-self-similar process with stationary increments and long memory, living in the $q$th Wiener chaos. The class of Hermite processes includes fractional Brownian motion ($q=1$), which is Gaussian, and the Rosenblatt process ($q=2$), which belongs to the second Wiener chaos. For any $m\geq 1$ and any grid $0<t_{1},...<t _{m}$, \cite{LNNT} shows that the determinant of the Malliavin matrix of the Hermite vector $Z= \left( Z^{H, q}_{t_{1}}, \ldots , Z^{H, q}_{t_{m}}\right)$ is almost surely strictly positive; therefore, $Z$ admits a density. 

The objective of the present work is to go a step further and investigate the smoothness of this density. As previously mentioned, smoothness is deeply linked to nondegeneracy,  and in particular to  the existence of negative moments of the determinant of the Malliavin matrix. While establishing  these negative moments appears challenging for general $q$, we succeed in proving  nondegeneracy for $q=2$, that is for Rosenblatt vectors. This relies crucially on the specific  structure of random variables in the second Wiener chaos, which can be represented as weighted sums of independent centered chi-square random variables.

Our first main result concerns the smoothness of the density of increments of the Rosenblatt process. 

\begin{theorem}\label{tt3}
	Let $(Z_{t}, t\in \mathbb{R})$ be a Rosenblatt process with self-similarity index $H\in \left( \frac{1}{2}, 1\right)$. Let $m\geq 1$ and $0=t_{0}<t_{1}<...<t_{m}$. Consider the random vector 
	\begin{equation}\label{zpi}
		Z_{\pi} =\left( Z_{t_{1}}-Z_{t_{0}}, \ldots , Z_{t_{m}}-Z_{t_{m-1}}\right),
	\end{equation}
	Then $Z_{\pi }$ admits a density belonging to the Schwartz space  $ \mathcal{S}(\mathbb{R} ^{m})$ of rapidly decreasing smooth functions. 
\end{theorem}

As a consequence, 
for all  $m\geq 1$ and for all $0\leq s_{1}<...< s_{m}$, the vector $(Z_{s_{1}},\ldots ,Z_{s_{m}})$ also has a density in $ \mathcal{S}(\mathbb{R} ^{m})$, see Corollary \ref{cor1}. The proof of Theorem \ref{tt3} relies on the classical Malliavin calculus theorem asserting that nondegenerate random vectors possess densities in the Schwartz space. 

Our second main result provides exponential-type bounds for the partial derivatives of the density  of the  finite-dimensional distributions of the Rosenblatt process. For $m \in \mathbb{N}^{\ast}$  and  $n=(n_{1},..., n_{m})\in \mathbb{N} ^{m}$, set 
\begin{equation*}
	\partial _{x} ^{n} =\prod_{j=1}^{m} \left( \frac{\partial}{\partial x_{j}}\right) ^{n_{j}}.
\end{equation*}

\begin{theorem}\label{tt2}
	Let $(Z_{t}, t\in \mathbb{R})$ be a Rosenblatt process with self-similarity index $H\in \left( \frac{1}{2}, 1\right)$.  For any  $m\geq 1$ and $n=(n_{1},..., n_{m})\in \mathbb{N}^{m}$, there exists $C>0$ (depending possibly on $m$ and $n$) such that for any grid $0=t_{0}<t_{1}<...<t_{m}$, the density $p_{\pi} $ of the vector $Z_{\pi}$  defined in (\ref{zpi})  satisfies
	\begin{equation*}
		\left| \partial _{x} ^{n} p_{\pi }(x)\right| \leq C \prod_{j=1}^{m} (t_{j}- t_{j-1}) ^{-H(1+ n_{j})} e ^{-c\frac{x_{j}}{(t_{j}- t_{j-1}) ^{H}}},
	\end{equation*}
	with $c=c(m,q)>0$, for every $x=(x_{1},..., x_{m}) $ such that $ x_{j}\geq 2(t_{j}-t_{j-1})^{H}$, $j=1,..., m.$ 
\end{theorem}

The proof of
Theorem \ref{tt2} is based on a Malliavin integration by parts formula expressing the partial derivatives of $p_{\pi}$ (see relation (\ref{22o-1}) in Proposition \ref{pp2} below)  and on a precise control of the resulting terms using Malliavin calculus and tail estimates for random variables in Wiener chaos.

The paper is organized  as follows. Section 2 provides the necessary preliminaries on Malliavin calculus, Wiener chaos and the Rosenblatt process. Section 3 contains the proof of the smoothness of the density of Rosenblatt vectors. Section 4 is devoted to the analysis of the regularity of the partial derivatives of this density. 

\section{Preliminaries}

In this preliminary section, we introduce the basic notions of Malliavin calculus and Wiener chaos that will be used throughout the paper. We also present the Rosenblatt process and recall its main properties.

\subsection{Wiener chaos, multiple stochastic integrals and Malliavin derivative}

We refer to the classical references \cite{N} and \cite{NP-book} for a comprehensive exposition of Wiener chaos and Malliavin calculus. Let $B=(B_{t}, t\in \mathbb{R})$ be a two-sided Brownian motion defined on a complete probability space $(\Omega,\mathcal{F},P)$. Set $\HHH=L^{2}(\mathbb{R})$, and for each $h\in\HHH$ define
\[
B(h)=\int_{\mathbb{R}} h(s)\,dB_{s},
\]
the Wiener integral of $h$ with respect to $B$. The family $\{B(h):h\in\HHH\}$ forms an isonormal Gaussian process, that is, a centered Gaussian family satisfying 
\[
\mathbf{E}[B(h)B(g)] = \langle h,g\rangle_{\HHH} \qquad \text{for all } h,g\in\HHH.
\]

We now recall the definition of the Wiener chaos generated by this isonormal process. Let $(H_q)_{q\ge0}$ denote the Hermite polynomials, defined recursively by  
\[
H_0(x)=1,\qquad H_1(x)=x,\qquad H_{q+1}(x)=xH_q(x)-qH_{q-1}(x).
\]
For $q\ge1$, the \emph{$q$th Wiener chaos} $\mathcal{H}_q$ is the closed linear span in $L^{2}(\Omega)$ of  
\[
\{ H_q(B(h)):\, h\in\HHH,\ \|h\|_{\HHH}=1 \}.
\]
In particular, $\mathcal{H}_1$ is the Gaussian chaos generated by $B(h)$, while $\mathcal{H}_q$ is non-Gaussian for every $q\ge2$.

\medskip

For $q\ge1$ and $h\in\HHH$ with $\|h\|_{\HHH}=1$, we set
\[
I_q(h^{\otimes q}) := H_q(B(h)).
\]
If $\|h\|\ne1$, we instead define 
\[
I_q(h^{\otimes q}) := \|h\|_{\HHH}^q\, H_q\!\left(B\!\left(\frac{h}{\|h\|_{\HHH}}\right)\right).
\]
By polarization, one extends the definition to elementary tensors $h_1\otimes\cdots\otimes h_q\in\HHH^{\otimes q}$, and then by linearity and density to any symmetric kernel $f\in\HHH^{\odot q}$. This yields a linear isometry
\[
I_q:\HHH^{\odot q}\to L^{2}(\Omega),
\]
satisfying
\[
\mathbf{E}[I_p(f)I_q(g)] =
\begin{cases}
p!\,\langle f,g\rangle_{\HHH^{\otimes p}}, & p=q,\\[1mm]
0, & p\ne q,
\end{cases}
\]
for all $f\in\HHH^{\odot p}$ and $g\in\HHH^{\odot q}$. One then obtains
\[
\mathcal{H}_q = \{ I_q(f): f\in\HHH^{\odot q}\}, \qquad q\ge1,
\]
that is, the $q$th Wiener chaos is precisely the image of $\HHH^{\odot q}$ under $I_q$.

\medskip

We now introduce the Malliavin derivative with respect to $B$. Let $\mathcal{S}$ denote the class of \emph{smooth cylindrical random variables} of the form 
\[
F = f(B(g_1),\dots,B(g_n)),
\]
where $n\ge1$, $g_1,\dots,g_n\in\HHH$, and $f\in C^\infty_P(\mathbb{R}^n)$ (i.e., smooth with all derivatives of polynomial growth). The \emph{Malliavin derivative} of $F$ is the process
\[
D_r F = \sum_{j=1}^n \frac{\partial f}{\partial x_j}(B(g_1),\dots,B(g_n))\, g_j(r),\qquad r\in\mathbb{R}.
\]
Higher-order derivatives $D^kF$ are defined by iteration. For $k,p\ge1$, the space $\mathbb{D}^{k,p}$ is the closure of $\mathcal{S}$ under the Sobolev-type norm
\begin{equation}\label{24o-1}
\|F\|_{k,p}^p
= \mathbf{E}[|F|^p] 
+ \sum_{j=1}^k \mathbf{E}\!\left[\|D^jF\|_{\HHH^{\otimes j}}^p\right],
\end{equation}
and we set $\mathbb{D}^\infty=\bigcap_{k,p\ge1}\mathbb{D}^{k,p}$. For $F,G\in\mathbb{D}^{1,2}$, the inner product of their derivatives is given by
\[
\langle DF, DG\rangle_{\HHH} = \int_{\mathbb{R}} D_rF\, D_rG\,dr.
\]

Let $\mathcal{U}$ be a separable Hilbert space. Denote by $\mathcal{S}_{\mathcal{U}}$ the class of $\mathcal{U}$-valued smooth random variables of the form $u=\sum_{j=1}^n F_j u_j$ with $u_j\in\mathcal{U}$ and $F_j\in\mathcal{S}$. One similarly defines the spaces $\mathbb{D}^{k,p}(\mathcal{U})$ and $\mathbb{D}^\infty(\mathcal{U})$ with associated norm
\[
\|u\|_{k,p,\mathcal{U}}^p
= \mathbb{E}\big[\|u\|_{\mathcal{U}}^p\big]
+ \sum_{j=1}^k \mathbb{E}\!\left[\|D^j u\|_{\HHH^{\otimes j}\otimes\mathcal{U}}^p\right].
\]

The \emph{Malliavin matrix} of $Z=(Z_1,\dots,Z_m)$ with $Z_i\in\mathbb{D}^{1,2}$ is the matrix $\Gamma_Z=(\Gamma_Z(i,j))_{1\le i,j\le m}$ defined by
\[
\Gamma_Z(i,j)=\langle DZ_i, DZ_j\rangle_{\HHH}.
\]

We now recall the notion of nondegeneracy used in Malliavin calculus.

\begin{definition}\label{def1}
Let $F=(F_1,\ldots,F_m)$ be an $m$-dimensional random vector with $F_i\in\mathbb{D}^\infty$ for all $i$. We say that $F$ is \emph{nondegenerate} if its Malliavin matrix $\Gamma_F$ is invertible and  
\[
(\det \Gamma_F)^{-1} \in \bigcap_{p\ge1} L^p(\Omega).
\]
\end{definition}

Before concluding this section, we recall two classical results that will play a key role in the sequel. The first is a tail estimate for random variables in a fixed Wiener chaos, corresponding to Theorem~6.7 in \cite{Jan}.

\begin{lemma}\label{ll11}
Let $X=I_q(f)$ with $f\in\HHH^{\odot q}$. Then there exists a universal constant $c_q>0$ such that for all $t\ge2$,
\[
P\!\left(
\frac{|X|}{(\mathbf{E}[|X|^2])^{1/2}}
\ge t
\right)
\le e^{-c_q\, t^{2/q}}.
\]
\end{lemma}

For any $\beta=(\beta_1,\ldots,\beta_j)\in\{1,\ldots,m\}^j$, let  
\[
\partial_{\beta_i} = \frac{\partial}{\partial x_{\beta_i}}, 
\qquad 
\partial_{\beta} = \partial_{\beta_1}\cdots\partial_{\beta_j}.
\]

We denote by $\mathcal{S}(\mathbb{R}^m)$ the Schwartz space of smooth functions $f:\mathbb{R}^m\to\mathbb{R}$ such that for every $k,j\ge1$ and every multi-index $\beta\in\{1,\ldots,m\}^j$,
\[
\sup_{x\in\mathbb{R}^m} |x|^k\,|\partial_{\beta}f(x)| <\infty.
\]
The next theorem provides a classical criterion ensuring that a random vector has a smooth density; see \cite{N}, Proposition~2.1.5.

\begin{theorem}\label{tt1}
If $F=(F_1,\ldots,F_m)$ is nondegenerate in the sense of Definition~\ref{def1}, then the law of $F$ admits a density belonging to the Schwartz space $\mathcal{S}(\mathbb{R}^m)$.
\end{theorem}

We conclude by recalling the following factorization of the determinant of the Malliavin matrix, proved in \cite{LNNT}, Lemma~4.1.

\begin{lemma}\label{ll1}
Let $Z=(Z_1,\ldots,Z_m)\in(\mathbb{D}^{1,1})^m$ and let $\Gamma_Z$ be its Malliavin matrix. Then, almost surely,
\[
\det \Gamma_Z
= \|DZ_1\|_{\HHH}^2
\prod_{j=2}^{m}
\big\| DZ_j - \mathrm{proj}_{E_{j-1}}(DZ_j)\big\|_{\HHH}^2,
\]
where $E_{j-1}$ is the closed linear span in $\HHH$ generated by $\{DZ_1,\ldots,DZ_{j-1}\}$ and $\mathrm{proj}_{E_{j-1}}$ denotes the orthogonal projection onto $E_{j-1}$.
\end{lemma}

\subsection{The Rosenblatt process}

We now introduce the Rosenblatt process and recall its main features. Let $H\in(1/2,1)$. The Rosenblatt process $(Z_t,t\in\mathbb{R})$ with self-similarity index $H$ is defined by
\begin{equation}\label{rose}
Z_t = I_2(L_t), \qquad t\in\mathbb{R},
\end{equation}
where the kernel $L_t$ is given, for $y_1,y_2\in\mathbb{R}$, by
\[
L_t(y_1,y_2)
= d(H)\int_0^t (u-y_1)_+^{\frac{H}{2}-1} \,(u-y_2)_+^{\frac{H}{2}-1}\,du.
\]
The constant $d(H)$ in \eqref{rose} is chosen so that $\mathbf{E}[(Z_t)^2]=|t|^{2H}$ for all $t\in\mathbb{R}$. We use the convention $\theta_+^\alpha=\theta^\alpha$ for $\theta>0$, and $0$ otherwise. Integrals over negative intervals are defined by $\int_0^t = -\int_t^0$ for $t<0$.

The Rosenblatt process is $H$-self-similar, i.e., for every $a>0$,
\[
(Z_{at}, t\in\mathbb{R}) \overset{law}{=} (a^H Z_t, t\in\mathbb{R}),
\]
and has stationary increments: for every $h\in\mathbb{R}$,
\[
(Z_{t+h}-Z_t, t\in\mathbb{R}) \overset{law}{=} (Z_t, t\in\mathbb{R}).
\]
Its covariance function is given by
\[
\mathbf{E}[Z_t Z_s]
= \tfrac12\big(|t|^{2H}+|s|^{2H}-|t-s|^{2H}\big),
\]
which coincides with that of fractional Brownian motion with Hurst index $H$. Since the Rosenblatt process is non-Gaussian, this covariance does not determine its law.

The random variable $Z_1$ is called the Rosenblatt random variable. If $0=t_0<t_1<\dots<t_m$, the random vector $(Z_{t_1}-Z_{t_0},\dots,Z_{t_m}-Z_{t_{m-1}})$ is called the \emph{Rosenblatt increment vector}, whereas for $0\le s_1<\dots<s_m$, the vector $(Z_{s_1},\dots,Z_{s_m})$ is called a \emph{Rosenblatt vector}.

Since $Z_t$ belongs to the second Wiener chaos for every $t\ne0$, it admits a series expansion as a weighted sum of independent centered chi-square random variables. More precisely (see Proposition~2.7.13 in \cite{NP-book}),
\begin{equation}\label{series}
Z_t = \sum_{j\ge1} \lambda_{j,t}\,(N_j^2-1),
\end{equation}
where $(N_j)_{j\ge1}$ is a sequence of i.i.d.\ $N(0,1)$ random variables, and the series \eqref{series} converges in $L^2(\Omega)$ and almost surely. The coefficients $\lambda_{j,t}$ can be expressed in terms of contractions of the kernel $L_t$.

\section{Smoothness of the density}

This section is devoted to the proof of Theorem \ref{tt3}. We first recall some auxiliary results, then establish the existence of negative moments for the determinant of the Malliavin matrix of a Rosenblatt vector, and finally deduce the smoothness of its density.

\subsection{Auxiliary results}

The two results recalled below were obtained in \cite{LNNT}. The first one asserts that a Rosenblatt increment vector admits a density, while the second provides a useful identity for the “Malliavin conditional variance”.

\begin{theorem}
	Let $(Z_{t}, t\in \mathbb{R})$ be a Rosenblatt process with self-similarity index $H\in \left( \frac{1}{2}, 1\right)$. Let $m\geq 2$ and $0=t_{0}<t_{1}<\dots<t_{m}$. Consider the random vector 
	\begin{equation*}
		Z_{\pi} =\left( Z_{t_{1}}-Z_{t_{0}},  Z_{t_{2}}-Z_{t_{1}}, \ldots ,  Z_{t_{m}}-Z_{t_{m-1}}\right).
	\end{equation*}
Then
\begin{equation*}
	P\left( \det \Gamma _{Z_{\pi} }>0\right)= 1.
\end{equation*}
In particular, the random vector $Z_{\pi}$ admits a density with respect to the Lebesgue measure on $\mathbb{R} ^{m}$. 
\end{theorem}

\noindent\textbf{Proof.} By Theorem 1.1 in \cite{LNNT}, the vector
\[
V:= (Z_{t_{1}},\dots, Z_{t_{m}})
\]
admits a density $f_{V}$ on $\mathbb{R}^{m}$. Consider the map $S:\mathbb{R}^{m}\to\mathbb{R}^{m}$ defined by 
\begin{equation}\label{s}
	S(y_{1},\dots, y_{m})= \left( y_{1},\, y_{1}+ y_{2},\, \dots,\, \sum _{j=1}^{m} y_{j} \right).
\end{equation}
Then $S$ is linear, invertible, and the determinant of its associated matrix is equal to one. We have $V=S(Z_{\pi})$. Since $S$ is a linear diffeomorphism, the law of $Z_{\pi}$ is the pushforward of the law of $V$ by $S^{-1}$. It follows that $Z_{\pi}$ admits a density given by $f_{Z_{\pi}}(y)= f_{V}(S(y))$ for all $y\in \mathbb{R}^{m}$. \qed

\begin{lemma}\label{ll2}
Let $(Z_{t}, t\in \mathbb{R})$ be a Rosenblatt process with self-similarity index $H\in \left( \frac{1}{2}, 1\right)$. Let $m\geq 2$ and $0=t_{0}<t_{1}<\dots<t_{m}$. For $j\geq 1$, set  
\begin{equation}
\label{21o-1}
\tilde{E}_{j-1}= \span \{ D(Z_{t_{k}}-Z_{t_{l}}): k,l=0,1,\dots, j-1\}.
\end{equation}
Then, for every $j \geq 1$,
\begin{equation*}
\left\| D(Z_{t_{j}}-Z_{t_{j-1}})- {\rm proj}_{\tilde{E}_{j-1}} \big(D(Z_{t_{j}}-Z_{t_{j-1}})\big)\right\| _{\HHH} ^{2} \overset{\rm law}{=} (t_{j}-t_{j-1})^{2H} F_{j},
\end{equation*}
where, for each $j\geq 1$, $F_{j}$ is a random variable such that $F_{j} \geq \Vert DZ_{1}\Vert ^{2}_{\HHH_{1}}$, with $\HHH_{1}= L ^{2}([0, 1])$.
\end{lemma}

\noindent\textbf{Proof.} This follows from the proof of Theorem 6.1 in \cite{LNNT}. \qed 

\subsection{Smoothness of the density: Proof of Theorem \ref{tt3}}

As a first step, we show that the norm of the Malliavin derivative of a Rosenblatt random variable admits negative moments of all orders.

\begin{proposition}\label{prop33}
Let $(Z_{t}, t\in \mathbb{R})$ be a Rosenblatt process. Then, for every $p\geq 1$, 
\begin{equation*}
\mathbf{E}\left[  \Vert DZ_{1}\Vert ^{-2p}_{\HHH _{1}}\right]<\infty. 
\end{equation*}
\end{proposition}

\begin{proof}
Following \eqref{rose}, we have \( Z_1 = I_2(L_1) \) for a symmetric kernel \( L_1 \in L^2(\mathbb{R}^2) \).
For each \( s \in [0,1] \), define
\[
g_s(\cdot) := L_1(s,\cdot) \in L^2(\mathbb{R}), 
\qquad 
Y_s := I_1(g_s).
\]
Then \((Y_s)_{s\in[0,1]}\) is a centered Gaussian process belonging to the first Wiener chaos, with covariance function
\[
K(s,t) = \mathbb{E}[Y_s Y_t] = \langle g_s, g_t \rangle_{L^2(\mathbb{R})}, 
\qquad s,t\in[0,1].
\]

\medskip
\noindent\textit{Step 1: Spectral decomposition of the covariance operator.}
Consider the Hilbert--Schmidt operator \( T:L^2([0,1]) \to L^2([0,1]) \) defined by
\[
(T\varphi)(s) := \int_0^1 K(s,t)\,\varphi(t)\,dt.
\]
The operator \(T\) is self-adjoint, positive, and compact. Hence, there exists an orthonormal basis 
\((\varphi_j)_{j\ge1}\) of \(L^2([0,1])\) consisting of eigenfunctions of \(T\), associated with nonnegative eigenvalues 
\((\lambda_j)_{j\ge1}\) satisfying \(\lambda_j \to 0\) as \(j\to\infty\). In particular,
\begin{equation}\label{eq:Kexp}
K = \sum_{j=1}^{\infty} \lambda_j\,\varphi_j \otimes \varphi_j.
\end{equation}

It is easy to verify that \(T = B^* B\), where
\[
B : L^2([0,1]) \to L^2(\mathbb{R}), 
\qquad 
Bf = \int_0^1 L_1(t,\cdot)\,f(t)\,dt,
\]
and
\[
B^* : L^2(\mathbb{R}) \to L^2([0,1]), 
\qquad 
B^*f = \int_{\mathbb{R}} L_1(s,\cdot)\,f(s)\,ds.
\]
From Step~2 in the proof of \cite[Theorem~1.2]{Rola}, we know that 
\(\operatorname{rank}(B_{|_{L^2(\mathbb{R})}}) = \infty.\)
Since
\[
\operatorname{rank}(B_{|_{L^2(\mathbb{R})}}) 
\le \operatorname{rank}(B) 
= \operatorname{rank}(B^*B)
= \operatorname{rank}(T),
\]
we deduce that \(\operatorname{rank}(T) = \infty\). 
Therefore, the set
\[
S := \{\, j \ge 1 : \lambda_j > 0 \,\}
\]
is infinite.

\medskip
\noindent\textit{Step 2: Construction of an orthogonal Gaussian family.}
For each \( j \in S \), define
\[
h_j := \int_0^1 \varphi_j(s)\,g_s\,ds \in L^2(\mathbb{R}),
\qquad
\xi_j := I_1(h_j),
\qquad
\zeta_j := \frac{\xi_j}{\sqrt{\lambda_j}}.
\]
A straightforward computation yields, for all \( i,j \in S \),
\[
\mathbb{E}[\xi_i \xi_j]
= \langle h_i, h_j \rangle_{L^2(\mathbb{R})}
= \int_0^1\!\!\int_0^1 \varphi_i(s)\varphi_j(t)K(s,t)\,ds\,dt
= \langle \varphi_i, T\varphi_j \rangle_{L^2([0,1])}
= \lambda_j\,\delta_{ij}.
\]
Hence, \((\xi_j)_{j\in S}\) are independent centered Gaussian random variables with 
\(\mathrm{Var}(\xi_j)=\lambda_j\), and consequently, 
\((\zeta_j)_{j\in S}\) are i.i.d.\ standard normal variables.

By the Karhunen--Loève expansion associated with \(K\),
\[
Y_s = \sum_{j\in S} \sqrt{\lambda_j}\,\zeta_j\,\varphi_j(s),
\qquad s\in[0,1],
\]
with convergence in \(L^2(\Omega\times[0,1])\).

\medskip
\noindent\textit{Step 3: Expression of the Malliavin derivative.}
Since \( D_s Z_1 = 2\,I_1(f_1(s,\cdot)) = 2\,Y_s \), we obtain
\[
D_s Z_1 = 2\sum_{j\in S} \sqrt{\lambda_j}\,\zeta_j\,\varphi_j(s),
\qquad s\in[0,1].
\]
By orthonormality of \((\varphi_j)\) in \(L^2([0,1])\),
\begin{equation}\label{eq:DZ1}
\|DZ_1\|_{L^2([0,1])}^2
= \int_0^1 (D_s Z_1)^2\,ds
= 4\sum_{j\in S} \lambda_j\,\zeta_j^2.
\end{equation}

\medskip
\noindent\textit{Step 4: Finite negative moments of \(\|DZ_1\|_{L^2([0,1])}\).}
Fix \(p \ge 1\), and let \( j_1, j_2, \dots \) denote an enumeration of \(S\). 
Since \(S\) is infinite and each \(\lambda_{j_k} > 0\), for any \(N \ge 1\),
\[
\|DZ_1\|_{L^2([0,1])}^2
\ge 4\sum_{k=1}^{N} \lambda_{j_k}\,\zeta_{j_k}^2
\ge 4\,\lambda_{j_N}\sum_{k=1}^{N}\zeta_{j_k}^2.
\]
Therefore,
\[
\|DZ_1\|_{L^2([0,1])}^{-2p}
\le (4\,\lambda_{j_N})^{-p}
\Big(\sum_{k=1}^{N}\zeta_{j_k}^2\Big)^{-p}.
\]
Since \(\sum_{k=1}^{N}\zeta_{j_k}^2 \sim \chi^2(N)\), its negative moments of order \(p\) are finite whenever \(N>2p\). 
Choosing \(N := 2p+1\) ensures this condition, and thus
\[
\mathbb{E}\!\left[\|DZ_1\|_{L^2([0,1])}^{-2p}\right] < \infty,
\qquad \forall\, p \ge 1.
\]
\end{proof}

We now deduce the existence of negative moments for the determinant of the Malliavin matrix of Rosenblatt increment vectors.

\begin{lemma}\label{ll3}
Let $(Z_{t}, t\in \mathbb{R})$ be a Rosenblatt process with self-similarity index $H\in \left( \frac{1}{2}, 1\right)$. Let $m\geq 1$ and $0=t_{0}<t_{1}<\dots<t_{m}$. Consider the random vector 
\begin{equation*}
Z_{\pi} =\left( Z_{t_{1}}-Z_{t_{0}}, \ldots , Z_{t_{m}}-Z_{t_{m-1}}\right),
\end{equation*}
and let $\Gamma_{Z_{\pi}}$ be the Malliavin matrix of $Z _{\pi }$. Then, for any integer $k\geq 0$ and any $p\geq 1$, 
\begin{equation}\label{kpnorm:invdet}
	\left\| \left( \det \Gamma_{Z_{\pi}}\right) ^{-1} \right\| _{k, p} \leq C \prod_{j=1}^{m} (t_{j}- t_{j-1}) ^{-2H}.
\end{equation}
\end{lemma}

\noindent\textbf{Proof.} We first consider the case $k=0$. We need to show that 
\begin{equation}\label{21o-2}
\left( \mathbf{E} \left[ \left( \det \Gamma_{Z_{\pi}}\right) ^{-p}\right]\right) ^{\frac{1}{p}} \leq C  \prod_{j=1}^{m} (t_{j}- t_{j-1}) ^{-2H}.
\end{equation}
We use Lemma \ref{ll1} to factorize $\det \Gamma_{Z_{\pi}}$. By this result, for every $p\geq 1$, 
\begin{equation*}
\left( \det \Gamma_{Z_{\pi}}\right) ^{-p}
=\Vert DZ_{t_{1}}\Vert _{\HHH} ^{-2p}
\prod_{j=2}^{m} 	\Vert D(Z_{t_{j}}-Z_{t_{j-1}})- {\rm proj}_{\tilde{E}_{j-1}} D(Z_{t_{j}}-Z_{t_{j-1}})\Vert _{\HHH} ^{-2p},
\end{equation*}
where $\tilde{E}_{j-1}$ is given by \eqref{21o-1}. Applying H\"older's inequality yields
\begin{eqnarray*}
\mathbf{E} \left[ 	\left( \det \Gamma_{Z_{\pi}}\right) ^{-p}\right]
&\leq& \left( \mathbf{E}\left[ \Vert DZ_{t_{1}}\Vert _{\HHH} ^{-2pm}\right]\right) ^{\frac{1}{m}} \\
&&\times \prod_{j=2} ^{m} \left( \mathbf{E}\left[ \Vert D(Z_{t_{j}}-Z_{t_{j-1}})- {\rm proj}_{\tilde{E}_{j-1}} D(Z_{t_{j}}-Z_{t_{j-1}})\Vert _{\HHH} ^{-2pm}\right]\right) ^{\frac{1}{m}}.
\end{eqnarray*}
By Lemma \ref{ll2}, for each $j=2,\dots,m$, 
\[
\Vert D(Z_{t_{j}}-Z_{t_{j-1}})- {\rm proj}_{\tilde{E}_{j-1}} D(Z_{t_{j}}-Z_{t_{j-1}})\Vert _{\HHH} ^{2}
\overset{ \rm law }{=}(t_{j}- t_{j-1} )^{2H}F_{j},
\]
with $F_{j}$ satisfying $F_{j}\geq \Vert DZ_{1}\Vert _{\HHH_{1}}^{2}$. Consequently, 
\begin{eqnarray*}
\mathbf{E} \left[ 	\left( \det \Gamma_{Z_{\pi}}\right) ^{-p}\right]
&\leq& \left( \mathbf{E}\left[ \Vert DZ_{t_{1}}\Vert _{\HHH} ^{-2pm}\right]\right) ^{\frac{1}{m}}
\prod_{j=2} ^{m} (t_{j}-t_{j-1}) ^{-2Hp}\left( \mathbf{E} \left[F_{j} ^{-pm}\right] \right) ^{\frac{1}{m}}\\
&\leq& \mathbf{E}\left[ \Vert DZ_{{1}}\Vert _{\HHH_{1}} ^{-2pm}\right]\prod_{j=1} ^{m} (t_{j}-t_{j-1}) ^{-2Hp}\\
&=& C \prod_{j=1} ^{m} (t_{j}-t_{j-1}) ^{-2Hp},
\end{eqnarray*}
and \eqref{21o-2} follows. 

Now let $k\geq 1$. By \eqref{24o-1}, 
\[
\left\| \det \Gamma _{Z_{\pi}}^{-1}  \right\|_{k,p}^p
= \mathbf{E} \left[ (\det \Gamma_{Z_{\pi}})^{-p}\right]
+ \sum _{j=1}^{k} \mathbf{E} \left[ \left\|D ^{j} (\det \Gamma_{Z_{\pi}})^{-1} \right\| ^{p}_{\HHH ^{\otimes j}}\right].
\]
By following exactly the same lines as in the proof of relation (4.11) in \cite{boufoussi2021local}, one can show that, for $j=1,\dots,k$, 
\[
\mathbf{E} \left[ \left\|D^j \big(\det \Gamma_{Z_{\pi}}\big)^{-1} \right\| ^{p}_{\HHH ^{\otimes j}}\right]
\leq  C  \prod_{i=1}^{m}  (t_{i}-t_{i-1})^{-2Hp},
\]
which yields the desired bound. \qed 

\medskip

\begin{remark}
Lemma~7.1 in \cite{Huetal} guarantees the existence of negative moments for the Malliavin derivative of general second-chaos random variables, but it does not yield quantitative, process-dependent estimates. In contrast, Lemma~\ref{ll3} (building on Proposition~\ref{prop33}) provides explicit bounds on all Sobolev norms of $(\det \Gamma_{Z_\pi})^{-1}$ in terms of the increments $(t_j - t_{j-1})$. This refined control is crucial for deriving the exponential-type estimates in Theorem~\ref{tt2}.
\end{remark}

\medskip

\noindent\textbf{Proof of Theorem \ref{tt3}.} The components of $Z_{\pi}$ clearly belong to $\mathbb{D} ^{\infty}$, since they are elements of the second Wiener chaos. Moreover, by Lemma \ref{ll3}, the Malliavin matrix $\Gamma_{Z_{\pi}}$ is invertible and $\left( \det \Gamma_{Z_{\pi}}\right) ^{-1} \in \bigcap_{p\geq 1}L ^{p} (\Omega)$. Thus, the random vector $Z_{\pi}$ is nondegenerate in the sense of Definition \ref{def1}. The conclusion then follows from Theorem \ref{tt1}. \qed

\medskip

We now deduce the smoothness of the density for Rosenblatt vectors.

\begin{corollary}\label{cor1}
Let $(Z_{t}, t\in \mathbb{R})$ be a Rosenblatt process with $H\in \left( \frac{1}{2}, 1\right)$. Let $m\geq 1$ and $0<s_{1}<\dots<s_{m}$. Then the vector $(Z_{s_{1}},\dots, Z_{s_{m}})$ admits a density which belongs to the Schwartz space $\mathcal{S}(\mathbb{R}^{m})$. 
\end{corollary}

\noindent\textbf{Proof.} Fix $0<t_{1}<\dots<t_{m}$ and set
\[
Z=(Z_{t_{1}}, \dots, Z_{t_{m}}).
\]
Then $Z= S(Z_{\pi})$, where $S$ is the map defined in \eqref{s}. Moreover, we have $f_{Z}(y)=f_{Z_{\pi}}(S^{-1}(y))$ for every $y\in \mathbb{R}^{m}$, where $f_{Z}$ and $f_{Z_{\pi}}$ denote the densities of $Z$ and $Z_{\pi}$, respectively.

It is a standard fact that if $f\in \mathcal{S}(\mathbb{R}^{m})$ and $A$ is an invertible linear map on $\mathbb{R}^{m}$, then $f\circ A\in \mathcal{S}(\mathbb{R}^{m})$. \qed

\section{Exponential-type upper bounds for the partial derivatives of the density}

This section is devoted to the proof of Theorem \ref{tt2}. The argument relies on two preliminary results. The first provides an estimate for the second moment of a random variable appearing in the representation of the partial derivatives of the density of a Rosenblatt increment vector. The second establishes an upper bound for the tail probability of a Rosenblatt increment vector.

\begin{proposition}\label{pp2}
Let $(Z_{t}, t\in \mathbb{R})$ be a Rosenblatt process with self-similarity index $H\in \left( \frac{1}{2}, 1\right)$. Let $n, m\geq 1$ be integers and fix $0=t_{0}<t_{1}<\dots<t_{m}$ and $\beta = \left( \beta_{1},\dots, \beta_{n} \right) \in \{1,\dots, m\}^{n}$. Let $Z_{\pi}$ be given by \eqref{zpi}. Then there exists a random variable $H ^{\beta} \in \mathbb{D} ^{\infty}$ such that, for every $\varphi \in  C_{p} ^{\infty}(\mathbb{R} ^{m})$,
\begin{equation}\label{22o-1}
\mathbf{E} \left[ \partial _{\beta } \varphi (Z_{\pi})\right] = \mathbf{E}\left[ \varphi (Z_{\pi})H ^{\beta}\right]. 
\end{equation}
Moreover,
\begin{equation}\label{22o-2}
\mathbf{E}\left[ (H ^{\beta}) ^{2} \right]\leq C \prod_{j=1}^{n}( t_{\beta_{j}}-t_{\beta _{j-1}})^{-2H}.
\end{equation}
\end{proposition}

\noindent\textbf{Proof.}
The existence of the random variable $H^{\beta}$ satisfying \eqref{22o-1} follows from Proposition 2.1.4 in \cite{N}. From Lemma 3.6 in \cite{boufoussi2021local}, we have the estimate
\begin{eqnarray}\label{eq:partial990088999} 
\left( \mathbf{E} \left[ (H ^{\beta}) ^{2}\right]\right) ^{\frac{1}{2}}
&\leq& C\left\|\left(\operatorname{det} \Gamma_{Z_{\pi}}\right)^{-1}\right\|_{n, 2^{n+2}}^n \times \prod_{j=1}^n\left\|D(Z_{t_{\beta_j}}-Z_{t_{\beta_{j-1}}})\right\|_{n, 2^{2(m+n)}, \HHH}\nonumber \\
&&\times  \prod_{i=1 \atop i\neq \beta_j}^m\left\|D(Z_{t_{i}}-Z_{t_{{i-1}}})\right\|_{n, 2^{2(m+n)}, \HHH}^2.
\end{eqnarray}
Since for each $t>0$, $Z_{t}$ belongs to the second Wiener chaos, we have, for $j=1,2$, 
\begin{equation*}
\mathbf{E} \left[ \Vert D ^{j} (Z_{t}- Z_{s} )\Vert ^{2}_{ \HHH ^{\otimes j}}\right]
= 2\,\mathbf{E}\left[  \vert Z_{t}- Z_{s} \vert ^{2}\right]
=  2\vert t-s\vert ^{2H} \mathbf{E} [Z_{1}^{2}],
\end{equation*}
and
\begin{equation*}
D^{j} (Z_{t}- Z_{s})=0 \qquad \text{for } j\geq 3. 
\end{equation*}
This implies that, for every $k,p\geq 1$ and every $0\leq s\leq t$,
\begin{equation}\label{25o-1}
\Vert D(Z_{t}- Z_{s})\Vert _{k, p, \HHH}\leq C (t-s) ^{H}.
\end{equation}
By plugging \eqref{25o-1} and \eqref{kpnorm:invdet} into \eqref{eq:partial990088999}, we obtain
\begin{eqnarray}
\left( \mathbf{E} \left[ (H ^{\beta}) ^{2}\right]\right) ^{\frac{1}{2}}
&\leq& C \prod_{k=1}^m (t_k-t_{k-1})^{-2nH} \prod_{j=1}^n \left( (t_{\beta_{j}}-t_{\beta_{j-1}})^{H}  \times  \prod_{i=1 \atop i \neq \beta_j}^m (t_{i}-t_{i-1})^{2H} \right)\nonumber\\
&=& C \prod_{k=1}^m (t_k-t_{k-1})^{-2nH} \prod_{j=1}^n (t_{\beta_{j}}-t_{\beta_{j-1}})^{-H}  \nonumber\\
&&\times\prod_{\theta=1}^n\left(  (t_{\beta_{\theta}}-t_{\beta_{\theta-1}})^{2H} \prod_{i=1 \atop i \neq \beta_\theta}^m (t_{i}-t_{i-1})^{2H}\right). \label{22o-3}
\end{eqnarray}
On the other hand, it is easy to see that 
\begin{eqnarray}\label{eq:partial99008899900088}
\prod_{\theta=1}^n \left(  (t_{\beta_{\theta}}-t_{\beta_{\theta-1}})^{2H} \prod_{i=1 \atop i \neq \beta_\theta}^m (t_{i}-t_{i-1})^{2H}\right)
= \prod_{k=1}^m (t_k-t_{k-1})^{2nH}.
\end{eqnarray}
Substituting \eqref{eq:partial99008899900088} into \eqref{22o-3}, we conclude that 
\begin{equation}\label{eq:HBeta01}
\left( \mathbf{E} \left[ (H ^{\beta}) ^{2}\right]\right) ^{\frac{1}{2}}\leq C \prod_{j=1}^{n} (t_{\beta_j}-t_{\beta_{j}-1}) ^{-H},
\end{equation}
which is the desired estimate. \qed

\medskip

The next result concerns the tail probability of a Rosenblatt increment vector.

\begin{proposition}
Let $(Z_{t}, t\in \mathbb{R})$ be a Rosenblatt process with self-similarity index $H\in \left( \frac{1}{2}, 1\right)$. Let $m\geq 1$ and $0=t_{0}<t_{1}<\dots<t_{m}$, and let $Z_{\pi}$ be given by \eqref{zpi}. Then, for every $x_{j} \in \left[ 2(t_{j}- t_{j-1})^{H}, \infty\right)$,
\begin{equation*}
P \left( \vert Z_{t_{j}}-Z_{t_{j-1}}\vert \geq x_{j}, j=1,\dots,m\right) \leq \prod_{j=1}^{m} \exp\!\left(- \frac{c x_{j}}{(t_{j}-t_{j-1})^{H}}\right),
\end{equation*}
where $c>0$ is a constant depending only on $q$ and $m$.
\end{proposition}

\noindent\textbf{Proof.}
By H\"older's inequality,
\begin{equation*}
P \left( \vert Z_{t_{j}}-Z_{t_{j-1}}\vert \geq x_{j}, j=1,\dots,m\right)
\leq \prod_{j=1}^{m}  P \left( \vert Z_{t_{j}}-Z_{t_{j-1}}\vert \geq x_{j} \right) ^{\frac{1}{m}}.
\end{equation*}
Next, we apply Lemma \ref{ll11} with $q=2$. For each $j=1,\dots,m$ and $x_{j}\geq 2(t_{j}-t_{j-1}) ^{H}$,
\begin{eqnarray*}
P \left( \vert Z_{t_{j}}-Z_{t_{j-1}}\vert \geq x_{j} \right)
&=& P \left( \frac {\vert Z_{t_{j}}-Z_{t_{j-1}}\vert }{(t_{j}-t_{j-1}) ^{H}}\geq \frac{x_{j}}{(t_{j}- t_{j-1}) ^{H}}\right) \\
&\leq & \exp\!\left(-c\frac{x_{j}}{(t_{j}- t_{j-1}) ^{H}}\right).
\end{eqnarray*}
Consequently,
\begin{eqnarray*}
P \left( \vert Z_{t_{j}}-Z_{t_{j-1}}\vert \geq x_{j}, j=1,\dots,m\right)
&\leq& \prod_{j=1}^{m} \left(  \exp\!\left(-c\frac{x_{j}}{(t_{j}- t_{j-1}) ^{H}}\right)\right) ^{\frac{1}{m}}\\
&=& \prod_{j=1}^{m} \exp\!\left(-\frac{c}{m}\frac{x_{j}}{(t_{j}- t_{j-1}) ^{H}}\right). 
\end{eqnarray*}
Absorbing the factor $1/m$ into the constant $c$ yields the claimed bound. \qed

\medskip

\noindent\textbf{Proof of Theorem \ref{tt2}.}
By Theorem \ref{tt3}, the random vector $Z_{\pi}$ admits a density $p_{\pi}$ in $C^{\infty}(\mathbb{R}^{m})$. By Proposition 2.1.5 in \cite{N} (see the first equation in its proof), we have, for every $x\in \mathbb{R}^{m}$,
\begin{equation*}
\partial_{x}^n p_{\pi }(x) =\partial _{\alpha} p_{\pi }(x) = (-1) ^{\vert \alpha \vert }\,\mathbf{E} \left[ \mathbbm{1}_{\{Z_{\pi} \geq x\}} H^{{\beta}}\right],
\end{equation*}
where
\begin{equation*}
\alpha=\big(\underbrace{1,\dots, 1}_{n_1\text{-times}},\ldots,\underbrace{m,\dots, m}_{n_m\text{-times}}\big)\qquad \text{and}\qquad
\beta=\big(\alpha,1,\ldots,m\big),
\end{equation*}
and $H ^{\beta}$ is a random variable in $\mathbb{D}^{\infty}$ satisfying the estimate \eqref{22o-2}. Thus, if $x=(x_{1},\dots, x_{m})$ with $x_{j}\geq 2(t_{j}- t_{j-1}) ^{H}$ for $j=1,\dots,m$, then by Cauchy--Schwarz,
\begin{eqnarray*}
\left| \partial _{x} ^{n} p_{\pi }(x)\right|
&\leq& P ( Z_{\pi}\geq x) ^{\frac{1}{2}} \left( \mathbf{E} [ (H ^{\beta}) ^{2}]\right) ^{\frac{1}{2}}\\
&\leq& C \prod_{j=1}^{m} \left(  \exp\!\left(-c\frac{x_{j}}{(t_{j}- t_{j-1}) ^{H}}\right)\right)\times \prod_{\theta=1}^{m+\sum_{i=1}^mn_i} (t_{\beta _{\theta}}-t_{\beta _{\theta}-1} )^{-H}\\
&=&  C \prod_{j=1}^{m} (t_{j}- t_{j-1}) ^{-H(1+ n_{j})} \exp\!\left(-c\frac{x_{j}}{(t_{j}- t_{j-1}) ^{H}}\right).
\end{eqnarray*}
This is precisely the desired estimate. \qed




\begin{thebibliography}{99}
	
	\bibitem{boufoussi2021local}
	\rm B. Boufoussi and Y.  Nachit  (2023): 
	\rm Local times for systems of non-linear
stochastic heat equations. 
\rm {\it Stochastics and Partial Differential Equations: Analysis and Computations}.  {\bf  11}(1), 388–425.
	
	\bibitem{Huetal}
	\rm Y. Hu, F. Lu and D. Nualart (2014): 
	\rm Convergence of densities of some functionals
	of Gaussian processes.
	\rm {\it Journal of Functional Analysis} {\bf 266}, 814–875.
	
	\bibitem{Jan}
	\rm S. Janson (1997): 
	\rm {\it  Gaussian Hilbert spaces, } 
	\rm Cambridge University Press. 
	
	
	\bibitem{LNNT}
	\rm L. Loosveldt, Y. Nachit, I. Nourdin and C. A. Tudor (2025): 
	\rm Absolute continuity of finite dimensional distributions of Hermite processes via Malliavin calculus.
	\rm Preprint. 


		\bibitem{NP-book}
\rm 	I. Nourdin and G. Peccati (2012): 
\rm Normal Approximations
	with Malliavin Calculus From Stein's Method to Universality.  
	\rm Cambridge
	University Press. 
	
	\bibitem{NNP}
	\rm I. Nourdin, D. Nualart and G. Poly (2013): 
	\rm Absolute continuity and convergence of densities for random vectors on Wiener chaos.  
	\rm {\it Electron. J. Probab.} {\bf 18}, 1--19.  
	
	
	\bibitem{N}
	\rm D. Nualart (2006):  
	\rm {\it The Malliavin Calculus and Related Topics}.  
	\rm Springer.
	
	
	\bibitem{NT}
	\rm D. Nualart and C. A. Tudor (2017): 
	\rm The determinant of the iterated Malliavin matrix and the density of a pair of multiple integrals.
	\rm {\it Ann. Probab.} {\bf 45}(1), 518--534.  
	
	 
	 	\bibitem{Rola}
	\rm R. Zintout  (2013): 
	 The total variation distance between two double Wiener-Itô integrals. 
	 \rm {\it Statist. Probab. Lett.} {\bf 83}(10), 2160–2167.

\end{thebibliography}


\section*{Acknowledgements}

I.N. gratefully acknowledges support from the Luxembourg National Research Fund (Grant No. O22/17372844/FraMStA).  
C. T.  acknowledges support from   the  ANR project SDAIM 22-CE40-0015,  MATHAMSUD grant 24-MATH-04 SDE-EXPLORE and from the Ministry of Research, Innovation and Digitalization (Romania), grant CF-194-PNRR-III-C9-2023.
Part of this work was carried out while Y. N. was at the University of Luxembourg under a grant funded by the European Union's Horizon 2020 research and innovation programme (Grant No. 811017). 

\end{document}